\documentclass[12pt]{article}
\usepackage{fullpage,comment,authblk,stackrel}
\usepackage[small,bf]{caption}
\usepackage{xparse,etoolbox}
\usepackage[nopar]{lipsum}

\usepackage{caption,subfigure,mathtools}

\usepackage{nicefrac}
\usepackage{hyperref}
\usepackage{amssymb,mathtools,amsthm}
\usepackage[all,cmtip]{xy}
\usepackage{eulervm,xfrac,mathdots}
\usepackage{rotating, setspace}

\newcommand{\blocktheorem}[1]{%
  \csletcs{old#1}{#1}
  \csletcs{endold#1}{end#1}
  \RenewDocumentEnvironment{#1}{o}
    {\par\addvspace{1.5ex}
     \noindent\begin{minipage}{\textwidth}
     \IfNoValueTF{##1}
       {\csuse{old#1}}
       {\csuse{old#1}[##1]}}
    {\csuse{endold#1}
     \end{minipage}
     \par\addvspace{1.5ex}}
}

\newcommand{\Mspace}        	{{\mathbb M}}

\newcommand{\Rspace}        	{{\mathbb R}}

\newcommand{\Nspace}        	{{\mathbb N}}
\newcommand{\Zspace}        	{{\mathbb Z}}

\newcommand{\mono}		{\hookrightarrow}

\newcommand{\Ccat}          	{{\mathsf{C}}}

\newcommand{\Finset}          	{{\mathsf{FinSet}}}

\newcommand{\Set}          	{{\mathsf{Set}}}
\newcommand{\Vect}          	{{\mathsf{Vec}}}

\newcommand{\End}          	{{\mathsf{End}}}

\newcommand{\Pmod}      	{{\mathsf{PMod}}}
\newcommand{\PDgm}      	{{\mathsf{PDgm}}}
\newcommand{\Dgm}          	{{\mathsf{Dgm}}}
\newcommand{\Rep}          	{{\mathsf{Rep}}}
\newcommand{\Ab}      	    	{{\mathsf{Ab}}}
\newcommand{\Rect}      	    	{{\mathsf{Rect}}}
\newcommand{\Finab}      	    	{{\mathsf{FinAb}}}

\newcommand{\Agroup}          	{{\mathcal{A}}}
\newcommand{\Bgroup}          	{{\mathcal{B}}}
\newcommand{\Ggroup}          	{{\mathcal{G}}}

\newcommand{\Igroup}          	{{\mathcal{I}}}

\newcommand{\image}		{\mathrm{im}}

\newcommand{\Dist}			{\mathsf{d}}

\newcommand{\field}			{\mathsf{k}}

\newcommand{\ee}			{\varepsilon}
\newcommand{\dd}			{\delta}

\newcommand\define[1]		{{\bf{#1}}}

\newcommand\module[1]		{{\mathsf{#1}}}	

\newcommand{\Grow}         	 {{\mathsf{Grow}}}

\makeatletter
\def\moverlay{\mathpalette\mov@rlay}
\def\mov@rlay#1#2{\leavevmode\vtop{%
   \baselineskip\z@skip \lineskiplimit-\maxdimen
   \ialign{\hfil$\m@th#1##$\hfil\cr#2\crcr}}}
\newcommand{\charfusion}[3][\mathord]{
    #1{\ifx#1\mathop\vphantom{#2}\fi
        \mathpalette\mov@rlay{#2\cr#3}
      }
    \ifx#1\mathop\expandafter\displaylimits\fi}
\makeatother

\newtheoremstyle{amit}
{7pt}
{7pt}
{}
{7pt}
{\bf}
{:}
{.5em}
{}

\begin{document}

\theoremstyle{amit}

\newcommand{\goodgap}{\hspace{\subfigtopskip}\hspace{\subfigbottomskip}}

\newtheorem{defn}{Definition}[section]
\newtheorem{defn_two}{Definition}[subsection]
\newtheorem{prop}{Proposition}[section]
\newtheorem{observe}{Observation}[section]
\newtheorem{corr}{Corollary}[section]
\newtheorem{lem}{Lemma}[section]
\newtheorem{thm}{Theorem}[section]
\newtheorem{rmk}{Remark}[section]
\newtheorem{ex}{Example}[section]
\newtheorem{ex_two}{Example}[subsection]

\blocktheorem{thm}

\renewcommand\Affilfont{\itshape\small}
\title{Generalized Persistence Diagrams
\thanks{This material is based upon work supported by the National Science Foundation under agreement
No.\ DMS-1128155. Any opinions, findings and conclusions or recommendations
expressed in this material are those of the author and do not necessarily reflect the views of the 
National Science Foundation.}
}
\author{Amit Patel}
\date{}
\affil{School of Mathematics \\ Institute for Advanced Study \\ Princeton, NJ 08540}
\maketitle

\begin{abstract}
We generalize the persistence diagram of Cohen-Steiner, Edelsbrunner, and Harer to the setting of constructible persistence
modules valued in a symmetric monoidal category.
We call this the \emph{type $\Agroup$} persistence diagram of a persistence module.
If the category is also abelian, then we define a second \emph{type $\Bgroup$} persistence diagram.
In addition, we show that both diagrams are stable to all sufficiently small perturbations of the module.
\end{abstract}

\section{Introduction}

Let $f : \Mspace \to \Rspace$ be a Morse function on a compact manifold $\Mspace$.
The function $f$ filters $\Mspace$ by sublevel sets $\Mspace_{f \leq r} = \{ x \in \Mspace \; | \; f(x) \leq r \}.$
Apply homology with coefficients in a field and we call the resulting object $\module{F}$ a \emph{constructible persistence module of vector spaces}.
The \emph{persistence diagram} and the \emph{barcode} are two invariants of a persistence module
obtained as follows.
	\begin{itemize}
	\item{By Images:}
	Edelsbrunner, Letscher, and Zomorodian \cite{Edelsbrunner2002} define the 
	\emph{persistent homology group} $\module{F}_s^t$, for $s < t$, as the image of $\module{F}(s < t)$.
	Cohen-Steiner, Edelsbrunner, and Harer \cite{CSEdH} define the \emph{persistence diagram} 
	of $\module{F}$ as a
	finite set of points in the plane above the diagonal satisfying the following property.
	For each $s < t$, the number of points in the upper-left quadrant defined by $(s,t)$ 
	is the rank of $\module{F}_s^t$.
	
	\item{By Indecomposables:} 
	The module $\module{F}$ is isomorphic to a finite direct sum of indecomposable persistence modules
	$\module{F} \cong \module{F}_1 \oplus \cdots \oplus \module{F}_n.$
	Any two ways of writing $\module{F}$ as a sum of indecomposables are the same up to
	a reordering of the indecomposables.
	Furthermore, each indecomposable $\module{F}_i$ is an \emph{interval persistence} module.
	That is, there are a pair of values $r < t$, where $t$ may be infinite, such that 
	$F_i(s)$ is a copy of the field for all values $r \leq s < t$ and zero elsewhere.
	\footnote{The interval persistence module $\module{F}_i$ is fully described by the half open 
	interval $[s,t)$.}
	Zomorodian and Carlsson define the \emph{barcode} of $\module{F}$ 
	as its list of indecomposables~\cite{ZC2005}.
	See also Carlsson and de Silva~\cite{Carlsson2010}.
	\end{itemize}
A barcode translates to a persistence diagram by plotting the left endpoint versus the right endpoint of each interval persistence module.
A persistence diagram translates to a barcode by turning each point $(s,t)$ in to an interval persistence module starting at $s$ and ending at $t$.
In this way, the persistence diagram is equivalent to a barcode.
However, the two definitions are very different in philosophy.

Suppose the homology of each sublevel set $\Mspace_{f \leq r}$ is calculated using
integer coefficients.
Then the resulting object $\module{F}$ is a \emph{constructible persistence module of finitely generated abelian groups}.
However, an indecomposable persistence module of finitely generated abelian groups need not look anything like an interval 
persistence module.
For example, the module in Figure \ref{fig:ab} is indecomposable.
Indecomposables are hard to interpret especially under perturbations to the module.

We generalize the persistence diagram of Cohen-Steiner, Edelsbrunner, and Harer to the setting of constructible persistence 
modules $\module{F}$ valued in a symmetric monoidal category $\Ccat$ with images.
The category of sets, the category of vector spaces, and the category of finitely generated abelian groups
are examples of such categories.
We call this diagram the \emph{type $\Agroup$ persistence diagram} of $\module{F}$.
If $\Ccat$ is also abelian, then we define a second \emph{type $\Bgroup$ persistence diagram} of $\module{F}$.
The category of vector spaces and the category of abelian groups are examples of abelian categories.
The type $\Bgroup$ persistence diagram of $\module{F}$ may contain less information than the type $\Agroup$
persistence diagram of $\module{F}$.
However, the advantage of a type $\Bgroup$ diagram is a stronger statement of stability.
Depending on $\Ccat$, our persistence diagrams may not be a complete invariant of a persistence module.

Persistence is motivated by data analysis and data is noisy.
A small perturbation to a persistence module should not result in a drastic change to its
persistence diagram.
We use the standard \emph{interleaving distance} to measure differences between persistence modules \cite{proximity}.
We define a new metric we call \emph{erosion distance} to measure differences between persistence
diagrams.
In Theorem \ref{thm:two}, we show that if the interleaving distance between two constructible
persistence modules valued in an abelian category $\Ccat$ is $\ee$, then the erosion
distance between their type $\Bgroup$ persistence diagrams is at most $\ee$.
We call this \emph{continuity} of type $\Bgroup$ persistence diagrams.
If $\Ccat$ is simply a symmetric monoidal category, then Theorem \ref{thm:one} is a weaker one-way 
statement of continuity for type $\Agroup$ persistence diagrams.
We call this \emph{semicontinuity} of type $\Agroup$ persistence diagrams.
These theorems show that the information contained in both diagrams is stable to 
all sufficiently small perturbations of the module.

Cohen-Steiner, Edelsbrunner, and Harer define a stronger metric on the set of persistence diagrams 
they call \emph{bottleneck distance}.
They show that for two Morse functions $f ,g : \Mspace \to \Rspace$, the bottleneck distance between their persistence diagrams is at most $\max |f-g|$.
They do this by looking at the $1$-parameter family of persistence modules 
obtained from the linear interpolation $h : \Mspace \times [0,1] \to \Rspace$ taking $h_0 = f$ to $h_1 = g$.
Using the Box Lemma, which is a local statement of stability, 
they track each point in the persistence diagram of $h_0$ all the way to the persistence diagram of $h_1$.
Theorem \ref{thm:two} resembles the Box Lemma and assuming $\Ccat$ has colimits,
there is a way to construct a $1$-parameter $1$-Lipschitz family of persistence modules between any two interleaved persistence modules \cite{interpolation}.
This suggests that bottleneck stability might extend to type $\Bgroup$ persistence diagrams.
We leave the issue of bottleneck stability for future investigations.

\section{Persistence Modules}
Let $(\Ccat, \Box)$ be an essentially small symmetric monoidal category with images.
By essentially small, we mean that the collection of isomorphism classes of objects in $\Ccat$ is a set.
A symmetric monoidal category is, roughly speaking, a category $\Ccat$ with a binary operation $\Box$ on its objects
and an identity object $e \in \Ccat$ satisfying the following properties:
	\begin{itemize}
	\item (Symmetry) $a \Box b \cong b \Box a$, for all objects $a, b \in \Ccat$
	\item(Associativity) $a \Box (b \Box c) \cong (a \Box b) \Box c$, for all objects $a,b,c \in \Ccat$
	\item (Identity) $a \Box e \cong a$, for all objects $a \in \Ccat$.
	\end{itemize}
See \cite[page 114]{K_book} for a precise definition of a symmetric monoidal category.
By images, we mean that for every morphism $f : a \to b$,
there is a monomorphism $h : z \to b$ and a morphism $g : a \to z$ such that $f = h \circ g$.
Furthermore, for a monomorphim $h' : z' \to b$ and a morphism $g' : a \to z'$ such that $f = h' \circ g'$,
there is a unique morphism $u : z \to z'$ such that the following diagram commutes:
	\begin{equation*}
	\xymatrix{
	a \ar[rr]^f \ar[rd]^g \ar[rdd]_{g'} && b \\
	& z \ar[ur]^h \ar[d]^{u} & \\
	& z'. \ar[uur]_{h'} &
	}
	\end{equation*}
See \cite[page 12]{mitchell} for a discussion of images.

\begin{defn}
\label{defn:module}
A {\bf persistence module} is a functor $\module{F} : (\Rspace, \leq) \to \Ccat$
out of the poset of real numbers.
\end{defn}

Let $S = \{s_1 < \cdots < s_n \}$ be a finite set of real numbers.
Let $e \in \Ccat$ be an identity object.

\begin{defn}
A persistence module $\module{F}$ is $S$-{\bf constructible} if
	\begin{itemize}
	\item for $p \leq q < s_1$, $\module{F}(p \leq q)$ is the identity on $e$
	\item for $s_i \leq p \leq q < s_{i+1}$, $\module{F}(p \leq q)$ is an isomorphism
	\item for $s_n \leq p \leq q$, $\module{F}(p \leq q)$ is an isomorphism.
	\end{itemize}
We say $\module{F}$ is \emph{constructible} if there is a finite set $S$ such that $\module{F}$ is $S$-constructible.
Note that if $\module{F}$ is $S$-constructible and $T$-constructible, then it is also $(S \cup T)$-constructible.
\end{defn}

We draw examples from the following five essentially small symmetric monoidal categories with images.

\begin{ex}
Let $\Finset$ be the category of finite sets.
$\Finset$ is a symmetric monoidal category under finite colimits (disjoint unions).
A constructible persistence module valued in this category is often called a \emph{merge tree} \cite{morozov13}.
\end{ex}

The following four categories have more structure: they are abelian (see \cite[page 124]{K_book}) and 
Krull-Schmidt (see Appendix \ref{sec:krull_schmidt}).
In short, an abelian category is a category that behaves like the category of abelian groups.
Finite products and coproducts are the same.
Every morphism has a kernel and a cokernel.
Every monomorphism is the kernel of some morphism, and every epimorphism is the cokernel of some morphism.
The symmetric monoidal operation $\Box$ is the direct sum $\oplus$.

\begin{ex}
Let $\Vect$ be the category of finite dimensional $\field$-vector spaces, for some fixed field $\field$.
Each vector space $a \in \Vect$ is isomorphic to $\field_1 \oplus \field_2 \oplus \cdots \oplus \field_n$, where $n$ is the dimension of $a$.
Note that every short exact sequence $0 \to a \to b \to c \to 0$ splits. 
That is, $b \cong a \oplus c$. 
\end{ex}

\begin{ex}
Let $\Ab$ be the category of finitely generated abelian groups.
An indecomposable of $\Ab$ is isomorphic to the infinite cyclic group $\Zspace$ or to a primary cyclic group
$\sfrac{\Zspace}{p^m \Zspace}$, for a prime $p$ and a positive integer $m$. 
By the fundamental theorem of finitely generated abelian groups, each
object is uniquely isomorphic to 
$$
\Zspace^n \oplus \frac{\Zspace}{p_1^{m_1} \Zspace} \oplus \frac{\Zspace}{p_2^{m_2}\Zspace} \oplus \cdots \oplus \frac{\Zspace}{p_k^{m_k}\Zspace},
$$
for some $n \geq 0$ and primary cyclic groups $\sfrac{\Zspace}{ p_i^{m_i} \Zspace}$.
Not every short exact sequence in this category splits.
Consider the following short exact sequence
	\begin{equation*}
	\xymatrix{
	0 \ar[r] & \dfrac{\Zspace}{2 \Zspace} \ar[r]^{\times 2} &  \dfrac{\Zspace}{4 \Zspace} \ar[r]^{/ } & \dfrac{\Zspace}{2\Zspace} \ar[r] & 0.
	}
	\end{equation*}
Of course $\sfrac{\Zspace}{ 4 \Zspace}$ is not isomorphic to $\sfrac{\Zspace}{ 2 \Zspace} \oplus \sfrac{\Zspace}{ 2 \Zspace}$.
A finitely generated abelian group is simple iff it is isomorphic to $\sfrac{\Zspace}{ p \Zspace}$ for $p$ prime.
That is, $\sfrac{\Zspace}{ p \Zspace}$ has no subgroups other than $0$ and itself.
\end{ex}

\begin{ex}
Let $\Finab$ be the category of finite abelian groups.
An indecomposable of $\Finab$ is isomorphic to a primary cyclic group $\Zspace / p^m \Zspace$, for
prime $p$ and a positive integer $m$.
By the fundamental theorem of finitely generated abelian groups, each
object is uniquely isomorphic to 
$$
\frac{\Zspace}{p_1^{m_1} \Zspace} \oplus \frac{\Zspace}{p_2^{m_2}\Zspace} \oplus \cdots \oplus \frac{\Zspace}{p_k^{m_k}\Zspace}.
$$
As shown in the previous example, not every short exact sequence in this category splits.
\end{ex}

\begin{ex}
Let $\Rep(\Nspace)$ be the category of functors from the commutative monoid of natural numbers 
$\Nspace = \{0, 1, \dots\}$ to $\Vect$.
We think of $\Nspace$ as a category with a single object and an endomorphism for each $n \in \Nspace$
where $n \circ m$ is $n+m$.
A functor in $\Rep(\Nspace)$ is completely determined by where it sends $1$.
$\Rep(\Nspace)$ is therefore equivalent to the category whose objects are endomorphisms
$A : a \to a$ in $\Vect$ and
whose morphisms $f : A \to B$ are maps $\hat{f} : a \to b$
such that the following diagram commutes:
	\begin{equation*}
	\xymatrix{
	a \ar[d]_A \ar[r]^{\hat{f}}& b \ar[d]^B \\
	a \ar[r]_{\hat{f}} & b.
	}
	\end{equation*}

We represent each object of $\Rep(\Nspace)$ by a square matrix of elements in $\field$.
Suppose $\field$ is algebraically closed.
Then such a matrix decomposes into a Jordan normal form
\begin{equation*}
\begin{pmatrix}
      J_1 &  &            \\
      & \ddots &   \\
      & &  J_n
\end{pmatrix}
\end{equation*}
where each Jordan block is of the form
\begin{equation*}
J_i = \begin{pmatrix}
      \lambda_i & 1 & &           \\
      & \lambda_i &  \ddots & \\
      & & \ddots & 1 \\
      & &   & \lambda_i
\end{pmatrix}.
\end{equation*}
The indecomposables of $\Rep(\Nspace)$ are Jordan blocks.
An object of $\Rep(\Nspace)$ is simple iff its a Jordan block of dimension one.

Not every short exact sequence in $\Rep(\Nspace)$ splits.
Let $A : \field \to \field$ be given by $(\lambda)$,
let $B : k^2 \to k^2$ be given by
$	\begin{pmatrix}
      	\lambda & 1  \\
     	 0 & \lambda
	\end{pmatrix}
$,
and let $f : A \to B$ be given by $\hat{f}(x) = (x,0)$.
The quotient $C = B / \image f$ is isomorphic to $A$.
This gives us a short exact sequence
	\begin{equation*}
	\xymatrix{
	0 \ar[r] & A \ar[r]^{f} &  B \ar[r]^{/ } & C \ar[r] & 0
	}
	\end{equation*}
that does not split because $B$ is not isomorphic to $(\lambda) \oplus (\lambda) = 
\begin{pmatrix}
      	\lambda & 0  \\
     	 0 & \lambda
	\end{pmatrix}$.
\end{ex}

Let $\Pmod(\Ccat)$ be the full subcategory of the functor category $\big[ (\Rspace, \leq), \Ccat \big]$ consisting of
constructible persistence modules.
Henceforth, all persistence modules are constructible.

\section{Interleaving Distance}
There is a natural distance between persistence modules.
For $\ee \in \Rspace$, let 
$$\mathsf{Shift}^\ee : (\Rspace, \leq) \to (\Rspace, \leq)$$
be the poset map that sends $r$ to $r+\ee$.
If $\module{F} \in \Pmod$ is $S$-constructible, then 
$\module{F} \circ \mathsf{Shift}^\ee$ is $(S+\ee)$-constructible. 
Thus $\mathsf{Shift}^\ee$ gives rise to a functor 
$$\Delta^\ee : \Pmod(\Ccat) \to \Pmod(\Ccat).$$
For each $\ee \geq 0$, there is a canonical morphism $\sigma^\ee_\module{F} : \module{F} \to \Delta^\ee(\module{F})$
given by $\sigma^\ee_\module{F}(r) = \module{F}(r \leq r+ \ee)$.

\begin{defn}
Two modules $\module{F}, \module{G} \in \Pmod(\Ccat)$ are $\ee$-\define{interleaved} if there are morphisms
$\phi : \module{F} \to \Delta^\ee(\module{G})$ and $\psi : \module{G} \to \Delta^\ee(\module{F})$ such that
$\sigma^{2\ee}_\module{F} = \Delta^\ee(\psi) \circ \phi$ and $\sigma^{2\ee}_\module{G} = \Delta^\ee(\phi)  \circ \psi$.
\end{defn}

Any two persistence modules $\module{F}$ an $\module{G}$ 
are constructible with respect to a common set $T = \{ t_1 < \cdots < t_m \}$.
Both $\module{F}$ and $\module{G}$ are therefore constant over the half-open intervals $[t_i, t_{i+1})$
and $[t_m, \infty)$.
As a consequence, if there is an interleaving between $\module{F}$ and $\module{G}$,
then there is a minimum interleaving between $\module{F}$ and $\module{G}$.

\begin{defn}
The {\bf interleaving distance} $\Dist_I(\module{F}, \module{G})$ between two persistence modules is the minimum
over all $\ee \geq 0$ such that $\module{F}$ and $\module{G}$ are $\ee$-interleaved.
If $\module{F}$ and $\module{G}$ are not interleaved, let $\Dist_I(\module{F}, \module{G}) = \infty$.
\end{defn}

\begin{ex}
Let $f : \Mspace \to \Rspace$ be a Morse function on a compact manifold $\Mspace$.
The function $f$ filters $\Mspace$ by sublevel sets $\Mspace_{f \leq r}$.
Apply homology with coefficients in $\field$ and the resulting object is in $\Pmod(\Vect)$.
Apply homology with integer coefficients and the resulting object is in $\Pmod(\Ab)$.
Apply homology with coefficients in a finite abelian group $G$ and the resulting object is in
$\Pmod(\Finab)$.
Suppose $\ee > |f-g|$.
Then $\Mspace_{f \leq r} \subseteq \Mspace_{g \leq r+\ee} \subseteq \Mspace_{f \leq r + 2 \ee}$
implying, by functoriality of homology, an $\ee$-interleaving between the two persistence modules.
\end{ex}

\begin{rmk}
The idea of interleavings appears in \cite{CSEdH} but it is not named until \cite{proximity}.
Since then, interleavings have been abstracted to other settings \cite{morozov13, BubScott, Curry2014, Bubenik2015,Les2015, deSilvaMunchPatel}.
\end{rmk}

\section{Persistence Diagrams}
We now generalize the persistence diagram of Cohen-Steiner, Edelsbrunner, and Harer.

\begin{defn}
Define $(\Dgm, \supseteq)$ as the poset of all half-open intervals $[q,r) \subset \Rspace$, for $q < r$, and all
half-infinite intervals $[q, \infty) \subset \Rspace$.
The poset relation is the containment relation.
\end{defn}

Let $S = \{ s_1 < \cdots < s_n \}$ be a finite set of real numbers and $\Ggroup$ an abelian group.
In the setting of Cohen-Steiner, Edelsbrunner, and Harer, the group $\Ggroup$ is the integers.

\begin{defn}
A map $X : \Dgm \to \Ggroup$ is \define{$S$-constructible} if for every $J \supseteq I$ such that $J \cap S = I \cap S$, $X(I) = X(J)$.
\end{defn}

We say a map $X : \Dgm \to \Ggroup$ is \emph{constructible} if it is $S$-constructible for some set $S$.
In the setting of Cohen-Steiner, Edelsbrunner, and Harer, $X$ is the rank function.

\begin{defn}
A map $Y : \Dgm \to \Ggroup$ is \define{$S$-finite} if $Y(I) \neq e$
implies $I = [s_i,s_j)$ or $I = [s_i,\infty)$.
\end{defn}

We say a map $Y : \Dgm \to \Ggroup$ is \emph{finite} if it is $T$-finite for some set $T$.

\begin{defn}
A \define{persistence diagram} is a finite map $Y : \Dgm \to \Ggroup$.
\end{defn}

We visualize the poset $\Dgm$ as the set of points in the extended plane 
$\Rspace \times \Rspace \cup \{\infty\}$ above the diagonal.
We visualize a persistence diagram $Y$ by marking each $I \in \Dgm$ for which $Y(I) \neq [e]$
with the group element $Y(I)$.
See Figures \ref{fig:finset}, \ref{fig:vect}, \ref{fig:ab}, \ref{fig:finab}, and \ref{fig:repn}.

In order to define a morphism between persistence diagrams, we require more structure on the abelian
group $\Ggroup$.
Let $(\Ggroup, \preceq)$ be an abelian group with a translation invariant partial ordering on its elements. 
That is if $a \preceq b$, then $a+c \preceq b+c$ for any $c \in \Ggroup$.
Let $e \in \Ggroup$ be the additive identity.

\begin{defn}
A \define{morphism} $Y_1 \to Y_2$ of persistence diagrams is the relation
	$$\sum_{J \in \Dgm: J \supseteq I} Y_1(J) \preceq \sum_{J \in \Dgm: J \supseteq I} Y_2(J),$$
for each $I \in \Dgm$ such that $Y_1(I) \neq e$.
\end{defn}

Let $\PDgm(\Ggroup)$ be the poset of persistence diagrams valued in $(\Ggroup, \preceq)$.

\begin{thm}[M\"obius Inversion Formula]
For any $S$-constructible map $X : \Dgm \to \Ggroup$,
there is an $S$-finite map $Y : \Dgm \to \Ggroup$ satisfying the M\"obius inversion formula
	$$X(I) = \sum_{J \in \Dgm: J \supseteq I} Y(J),$$
for each $I \in \Dgm$.
\end{thm}

\begin{proof}
Let $S = \{ s_1 < \cdots < s_n \}$.
Define
	\begin{align}
	\label{eq:square} Y\big([s_i,s_j)\big) &= X\big([s_i,s_j)\big) - X\big([s_i,s_{j+1})\big) + 
	X\big([s_{i-1},s_{j+1})\big) - 
	X\big([s_{i-1},s_j)\big) \\
	\label{eq:halfsquare} Y\big([s_i,\infty)\big) &= X\big([s_i,\infty)\big) - X\big([s_{i-1},\infty)\big).
	\end{align}
Here we interpret $s_0$ as any value less than $s_1$ and $s_{n+1}$ as any value greater than $s_{n}$.
Define $Y(I) = e$ for all other $I \in \Dgm$.
Let us check that $Y$ satisfies the M\"obius inversion formula.
Fix an interval $I \in \Dgm$. 
Suppose $I = [s_i,s_j)$.
We have
	\begin{align*}
	\sum_{J \in \Dgm : J \supseteq I} Y(J) &=  \sum_{k = j}^{n} \sum_{h=1}^{i} Y\big([s_h,s_k)\big) + \sum_{h=1}^{i} Y\big([s_h,\infty)\big) \\
	&= \sum_{k = j}^{n} \sum_{h=1}^{i} \Big[ X\big([s_h,s_k)\big) - X\big([s_h,s_{k+1})\big) + X\big([s_{h-1},s_{k+1})\big) - 
	X\big([s_{h-1},s_k)\big) \Big] \\
	&+ \sum_{h=1}^{i} \Big[ X\big([s_h,\infty)\big) - X\big([s_{h-1},\infty)\big) \Big] \\
	&= \sum_{k=j}^{n} \Big[ X\big([s_i,s_k)\big) - X\big([s_i, s_{k+1})\big) \Big] + X\big([s_i,\infty)\big) \\
	&= X\big([s_i,s_j)\big).
	\end{align*}
Suppose $I$ is of the form $[s_i,\infty)$.
We have
	\begin{align*}
	\sum_{J \in \Dgm : J \supseteq I} Y(J) &=  \sum_{h=1}^{i} Y \big( [s_h,\infty) \big) \\
	&= \sum_{h=1}^{i} \big[ X\big([s_h,\infty)\big) - X\big([s_{h-1},\infty)\big) \big] \\
	&= X\big([s_i,\infty)\big).
	\end{align*}
Suppose $I$ is not of the form $[s_i,s_j)$.
Then there is an $I' \in \Dgm$ of the form $[s_i,s_j)$ or $[s_i, \infty)$ such that $I' \cap S = I \cap S$.
We have
	$$\sum_{J \in \Dgm : J \supseteq I} Y(J) = \sum_{J \in \Dgm: J \supseteq I'} Y(J) = X\big(I'\big) = X(I).$$
\end{proof}

The persistence diagram $Y$ of Cohen-Steiner, Edelsbrunner, and Harer is the 
M\"obius inversion of the rank function $X$.

\begin{rmk}
The M\"obius inversion formula applies to any constructible map from a poset to an abelian group.
See \cite{rota64,BenderGoldman75,leinster2012}.
This suggests a notion of a persistence diagram for constructible persistence modules not just over $(\Rspace, \leq)$
but over more general posets.
See \cite{BubScott,Bubenik2015}.
\end{rmk}

\section{Erosion Distance}
The interleaving distance suggests a natural metric between persistence diagrams.
For $\ee \geq 0$, let
	$$\Grow^\ee : \Dgm \to \Dgm$$
be the poset map that sends each $[p,q)$ to $[p-\ee,q+\ee)$ and each $[p,\infty)$ to $[p-\ee, \infty)$.
For a morphism $Y_1 \to Y_2$ in $\PDgm(\Ggroup)$, we have $Y_1 \circ \Grow^\ee \to Y_2 \circ \Grow^\ee$.
Thus $\Grow^\ee$ gives rise to a functor
	$$\nabla^\ee : \PDgm(\Ggroup) \to \PDgm(\Ggroup)$$
given by precomposition with $\Grow^\ee$.
For each $\ee \geq 0$, we have $\nabla^\ee(Y) \to Y$.
The persistence diagram $\nabla^\ee(Y)$ is visualized as the persistence diagram $Y$
with all its points shifted towards the diagonal by a distance $\sqrt{2} \ee$.
See Figure \ref{fig:erode}.
	\begin{figure}
	\centering
	\includegraphics[scale=0.75]{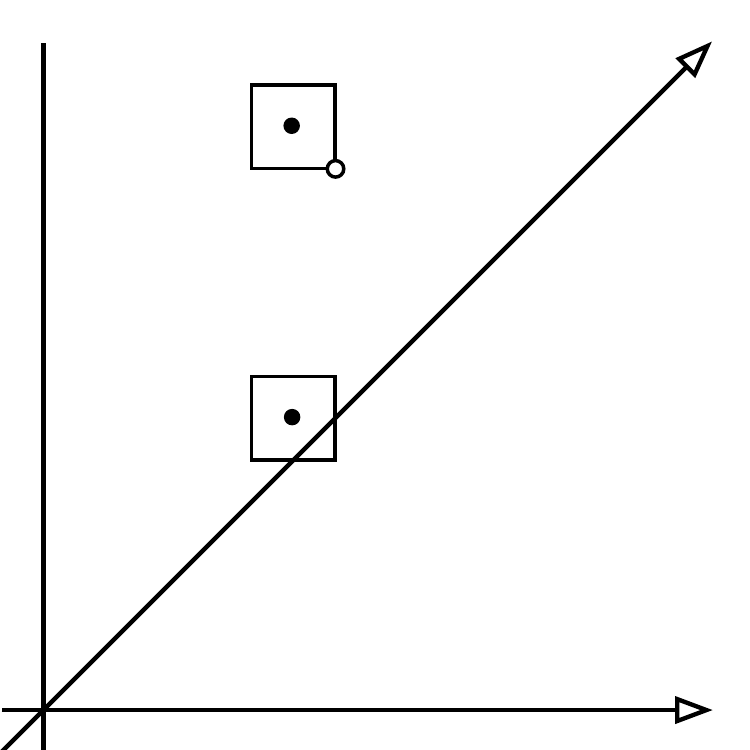} 
	\caption{The $\ee$-erosion $\nabla^\ee(Y)$ (circle) of a persistence diagram $Y$ (dots) slides each point of $Y$ 
	to the lower-right corner of the square of side length $2\ee$ centered at that point.
	Points close to the diagonal disappear into the diagonal.
	Note that $\nabla^\ee(Y) \to Y$.}
	\label{fig:erode}
	\end{figure}

\begin{defn}
An \define{$\ee$-erosion} between two persistence diagrams $Y_1, Y_2, \in \PDgm(\Ggroup)$ 
is a pair of morphisms $\nabla^\ee(Y_2) \to Y_1$ and $\nabla^\ee(Y_1) \to Y_2$.
\end{defn}

Any two persistence diagrams are finite with respect to a common set $T = \{t_1 < \cdots < t_n\}$.
As a consequence, if there is an $\ee$-erosion between $Y_1$ and $Y_2$, then there is a minimum $\ee$
for which there is an $\ee$-erosion.

\begin{defn}
The \define{erosion distance} $\Dist_E(Y_1, Y_2)$ is the minimum over all $\ee \geq 0$ such that
there is an $\ee$-erosion between $Y_1$ and $Y_2$.
If there is no $\ee$-erosion, let $\Dist_E(Y_1,Y_2) = \infty$.
\end{defn}

\begin{prop}
Let $X : \Dgm \to \Ggroup$ be a constructible map and let $Y : \Dgm \to \Ggroup$ be a finite map
that satisfies the M\"obius inversion formula
	$$X(I) = \sum_{J \in \Dgm: J \supseteq I} Y(J),$$
for each $I \in \Dgm$.
Then $$X \circ \Grow^\ee(I) = \sum_{J \in \Dgm: J \supseteq I} \nabla^\ee (Y)(J),$$
for each $I \in \Dgm$. 
In other words, $\Grow^\ee$ commutes with the M\"obius inversion formula.
\end{prop}
\begin{proof}
We have
	\begin{equation*}
	\begin{split}
	\sum_{J \in \Dgm: J \supseteq I} \nabla^\ee (Y)(J) &=
	\sum_{J \in \Dgm: J \supseteq I} Y \circ \Grow^\ee(J) \\
	&= X \circ \Grow^\ee(I)
	\end{split}
	\end{equation*}
\end{proof}

\begin{rmk}
The erosion distance first appears in \cite{Edelsbrunner2011} which is an early attempt to develop
a theory of persistence for maps from a surface to the Euclidean plane.
\end{rmk}

\section{Grothendieck Groups}

We are interested in two abelian groups: the Grothendieck group $\Agroup$ of an essentially small
symmetric monoidal category and the Grothendieck group $\Bgroup$ of an essentially small abelian category.
See \cite{K_book} for an introduction to the two Grothendieck groups.
Note that every abelian category is a symmetric monoidal category under the direct sum $\oplus$
and the additivity identity is the zero object.

\subsection{Symmetric Monoidal Category}
Let $\Ccat$ be an essentially small monoidal category.
The set $\Igroup(\Ccat)$ of isomorphism classes in $\Ccat$ is a commutative monoid under $\Box$.
We write the isomorphism class of an object $a \in \Ccat$ as $[a] \in \Igroup(\Ccat)$,
the binary operation in $\Igroup(\Ccat)$ as $[a] + [b] = [a \Box b]$, and the additive identity of $\Igroup(\Ccat)$ as $[e]$.

\begin{defn_two}
The \define{Grothendieck group $\Agroup(\Ccat)$} of $\Ccat$
is the group completion of the commutative monoid $\Igroup(\Ccat)$.
\end{defn_two}

Explicitly, an element of $\Agroup(\Ccat)$ is of the form $[a] - [b]$ with addition coordinatewise, and $[a] = [c]$ iff $[a] + [d] = [c] + [d]$,
for some element $[d] \in \Igroup(\Ccat)$.
If $\Ccat$ is additive and Krull-Schmidt (see Appendix \ref{sec:krull_schmidt}), then 
each object in $\Ccat$ is isomorphic to a unique direct sum of indecomposables.
This means $\Agroup(\Ccat)$ is the free abelian group generated by the
set of isomorphism classes of indecomposables.
The Grothendieck group $\Agroup(\Ccat)$ has a natural translation-invariant partial ordering.
We define $[a] \preceq [b]$ iff $[b]-[a] \in \Igroup(\Ccat)$.
If $[a] \preceq [b]$, then $[a]+[c] \preceq [b]+[c]$ for any $[c] \in \Agroup(\Ccat)$.
See \cite[page 72]{K_book} for an introduction to translation-invariant partial orderings on Grothendieck groups.

\begin{ex_two}
Every finite set is a finite disjoint union of the singleton set.
We have $$\Agroup(\Finset) \cong \Zspace.$$
\end{ex_two}

\begin{ex_two}
Every finite dimensional vector space is isomorphic to a finite direct sum of $\field$.
We have $$\Agroup(\Vect) \cong \Zspace.$$
\end{ex_two}

\begin{ex_two}
An indecomposable of $\Ab$ is the free cyclic group or a primary cyclic group.
We have $$\Agroup(\Ab) \cong \Zspace \oplus \bigoplus_{(m,p)} \Zspace,$$ 
over all primes $p$ and positive integers $m$.
\end{ex_two}

\begin{ex_two}
An indecomposable of $\Finab$ is a primary cyclic group.
We have
	$$\Agroup(\Finab) \cong \bigoplus_{(m,p)} \Zspace$$
over all primes $p$ and positive integers $m$.
\end{ex_two}

\begin{ex_two}
An indecomposable of $\Rep(\Nspace)$ is a Jordan block.
We have $$\Agroup \big( \Rep(\Nspace) \big) \cong \bigoplus_{(m,\lambda)} \Zspace,$$
over all positive integers $m$ and elements $\lambda$ in the field $\field$.
\end{ex_two}

\subsection{Abelian Category}
Suppose $\Ccat$ is an essentially small abelian category.
We say two elements $[b]$ and $[a] + [c]$ in $\Agroup(\Ccat)$ are related, written $[b] \sim [a] + [c]$,
if there is a short exact sequence $0 \to a \to b \to c \to 0$.

\begin{defn_two}
The \define{Grothendieck group $\Bgroup(\Ccat)$} of $\Ccat$  
is the quotient group $\Agroup(\Ccat) / \sim$.
That is, $\Bgroup(\Ccat)$ is the abelian group with one generator for each isomorphism classes $[a]$ in $\Ccat$
and one relation $[b] \sim [a] + [c]$ for each short exact sequence $0 \to a \to b \to c \to 0$.
\end{defn_two}

Let $\pi : \Agroup(\Ccat) \to \Bgroup(\Ccat)$ be the quotient map.
Note that $\pi \big( \Igroup(\Ccat) \big)$ is a commutative monoid that generates $\Bgroup(\Ccat)$.
This allows us to define a translation-invariant partial ordering on $\Bgroup(\Ccat)$ as follows.
We define $[a] \preceq [b]$ iff $[b] - [a] \in \pi \big( \Igroup(\Ccat) \big)$.
If $[a] \preceq [b]$, then $[a] + [c] \preceq [b] + [c]$ for any $[c] \in \Bgroup(\Ccat)$.
The quotient map $\pi$ is a poset map.

\begin{ex_two}
Every short exact sequence in $\Vect$ splits.
We have $$\Bgroup(\Vect) \cong \Zspace.$$
The quotient map $\pi : \Agroup(\Vect) \to \Bgroup(\Vect)$ is the identity.
\end{ex_two}

\begin{ex_two}
Every primary cyclic group $\sfrac{\Zspace}{p^m \Zspace}$ fits into a short exact sequence
	$$0 \to \Zspace \to \Zspace \to \frac{\Zspace}{p^m \Zspace}  \to 0.$$
This means $[\Zspace] \sim [\Zspace] + \big[\frac{\Zspace}{p^m \Zspace} \big]$
and therefore $0 \sim \big[\frac{\Zspace}{p^m \Zspace}\big]$.
We have $$\Bgroup(\Ab) \cong \Zspace.$$
The quotient map $\pi : \Agroup(\Ab) \to \Bgroup(\Ab)$ forgets the torsion part of every finitely generated abelian group.
\end{ex_two}

\begin{ex_two}
Every primary cyclic group $\sfrac{\Zspace}{p^m \Zspace}$ fits into a short exact sequence
	$$0 \to \frac{\Zspace}{p \Zspace} \to \frac{\Zspace}{p^m \Zspace} \to \frac{\Zspace}{p^{m-1} \Zspace} \to 0.$$ 
This means $$\left[ \frac{\Zspace}{p^m \Zspace} \right] \sim m \left[ \frac{\Zspace}{p \Zspace} \right].$$
Furthermore, $\frac{\Zspace}{p \Zspace}$ is a simple object so it can not be broken by a short exact sequence.
We have $$\Bgroup(\Finab) \cong \bigoplus_{p} \Zspace$$
over all $p$ prime.
The quotient map $\pi : \Agroup(\Finab) \to \Bgroup(\Finab)$ takes each primary cyclic group $\left[ \frac{\Zspace}{p^m \Zspace} \right]$ to
$m$ in the $p$ factor of $\Bgroup(\Finab)$.
\end{ex_two}

\begin{ex_two}
Every Jordan block fits into a short exact sequence.
For example,
	$$0 \to
	(\lambda) \to  
	\begin{pmatrix}
      	\lambda & 1  & 0 \\
     	 0 & \lambda & 1 \\
	 0 & 0 & \lambda
	\end{pmatrix}
	\to 
	\begin{pmatrix}
      	\lambda & 1  \\
     	 0 & \lambda
	\end{pmatrix} \to 0$$
	and
	$$0 \to (\lambda) \to  
	\begin{pmatrix}
      	\lambda & 1  \\
     	 0 & \lambda
	\end{pmatrix}
	\to (\lambda) \to 0.$$
This means $$\begin{pmatrix}
      	\lambda & 1  & 0 \\
     	 0 & \lambda & 1 \\
	 0 & 0 & \lambda
	\end{pmatrix}
	\sim 3 (\lambda).$$
Futhermore, each one-dimensional Jordan block $(\lambda)$ is simple so it can not be broken
by a short exact sequence.
We have 
	$$\Bgroup \big( \Rep(\Nspace) \big) \cong \bigoplus_{\lambda \in \field} \Zspace.$$
The quotient map $\pi : \Agroup\big(\Rep(\Nspace)\big) \to \Bgroup\big(\Rep(\Nspace)\big)$ 
takes each Jordan block of dimension $m \in \Nspace$
with eigenvalue $\lambda \in \field$ to $m$ in the $\lambda$ factor of $\Bgroup \big( \Rep(\Nspace) \big)$.
\end{ex_two}

\section{Diagram of a Module}
Fix an essentially small symmetric monoidal category $\Ccat$ with images.
We now assign to each persistence module $\module{F} \in \Pmod(\Ccat)$ a persistence diagram 
$\module{F}_\Agroup \in \PDgm \big( \Agroup(\Ccat) \big )$.
If $\Ccat$ is also abelian, then we assign to $\module{F}$ a second persistence diagram 
$\module{F}_\Bgroup \in \PDgm \big( \Bgroup(\Ccat) \big )$.

We start by constructing a map
	$$d \module{F}_\Igroup : \Dgm \to \Igroup(\Ccat).$$
Recall $\Igroup(\Ccat)$ is the commutative monoid of isomorphism classes of objects in $\Ccat$.
Suppose $\module{F}$ is $S = \{ s_1 < \cdots < s_n \}$-constructible.
Then there is a $\delta > 0$ such that $s_{i-1} < s_i - \delta$, for each $1 < i \leq n$.
Choose a value $s' > s_n$.
Define
	\begin{equation*}
	d \module{F}_{\Igroup}(I) = 
	\begin{cases}
	\big[ \image\; \module{F}(p < s_i - \dd) \big] & \text{for}\;  I = [p, s_i) \\
	\big[ \image\; \module{F}(p < s') \big] & \text{for}\; I = [p, \infty) \\
	\big[ \image\; \module{F}(p < q) \big] & \text{for all other}\; I = [p,q).
	\end{cases}
	\end{equation*}
Note that if $\module{F}$ is also $T$-constructible, then $d \module{F}_\Igroup$ 
constructed using $T$ is the same as $d \module{F}_\Igroup$ constructed using $S$.
Now compose with the inclusion map $\Igroup(\Ccat) \mono \Agroup(\Ccat)$ and we have an $S$-constructible
map 
$$d\module{F}_\Agroup : \Dgm \to \Agroup(\Ccat).$$
Suppose $\Ccat$ is abelian.
Then by composing with the quotient map $\pi : \Agroup(\Ccat) \to \Bgroup(\Ccat)$, we have an $S$-constructible
map
$$d\module{F}_\Bgroup : \Dgm \to \Bgroup(\Ccat).$$

\begin{defn}
The \define{type $\Agroup$ persistence diagram} of $\module{F}$ is the M\"obius inversion 
	$$\module{F}_\Agroup : \Dgm \to \Agroup(\Ccat)$$
of $d\module{F}_\Agroup : \Dgm \to \Agroup(\Ccat)$.
\end{defn}

\begin{defn}
The \define{type $\Bgroup$ persistence diagram} of $\module{F}$ is the M\"obius inversion 
	$$\module{F}_\Bgroup : \Dgm \to \Bgroup(\Ccat)$$
of $d\module{F}_\Bgroup : \Dgm \to \Bgroup(\Ccat)$.
\end{defn}

Note that if $\module{F}$ is $S$-constructible, then both $\module{F}_\Agroup$ and $\module{F}_\Bgroup$ are 
$S$-finite persistence diagrams.

\begin{prop}[Positivity]
\label{prop:pos}
For each $I \in \Dgm$, $[e] \preceq \module{F}_\Bgroup(I)$.
\end{prop}
\begin{proof}
Suppose $\module{F}$ is $S = \{ s_1 < \cdots < s_n \}$-constructible.
We need only show the inequality for intervals $I$ of the form $[s_i,s_j)$ and $[s_i, \infty)$.
For all other $I$, $\module{F}_\Bgroup(I) = [e]$.

Suppose $I = [s_i,s_j)$.
Consider the following subdiagram of $\module{F}$, for a sufficiently small $\dd > 0$:
\begin{equation*}
\xymatrix{
\module{F}(s_{i-1}) \ar[d]^{} \ar[rrr]^{\module{F}(s_{i-1} < s_i )} &&& \module{F}(s_i)  \ar[d]^{\module{F}(s_i < s_j -\dd)} \\
\module{F}(s_{j+1}-\dd)  &&& \module{F}(s_j-\dd). \ar[lll]^{\module{F}(s_j -\dd < s_{j+1}-\dd)}
}
\end{equation*}
Here we interpret $s_0$ as any value less than $s_1$ and $s_{n+1}$ as any value greater than $s_{n}$.
By Equation~\ref{eq:square},
\begin{equation*}
\module{F}_\Bgroup\big( [s_i,s_j) \big) = d\module{F}_\Bgroup\big( [s_i,s_{j}) \big) - d \module{F}_\Bgroup \big( [s_i,s_{j+1}) \big) + 
d\module{F}_\Bgroup \big( [s_{i-1},s_{j+1}) \big) - d\module{F}_\Bgroup \big( [s_{i-1},s_j) \big) \\
\end{equation*}
Observe
	\begin{equation*}
	\begin{split}
	d\module{F}_\Bgroup \big( [s_i,s_{j}) \big) - d \module{F}_\Bgroup \big( [s_i,s_{j+1}) \big) =&
	\left [ \image\ \module{F}(s_i < s_j -\dd) \right ] \\
	&- \left [ \frac{\image\ \module{F}(s_i < s_j -\dd)}
	{\image\ \module{F}(s_{i} < s_j -\dd ) \cap \ker\ \module{F}(s_j -\dd < s_{j+1}-\dd)} \right] \\
	=& \left [ \image\ \module{F}(s_i < s_j -\dd) \right ] - \left [ \image\ \module{F}(s_i < s_j -\dd) \right ] \\
	&+ \left[ \image\ \module{F}(s_i < s_j -\dd) \cap \ker\ \module{F}(s_j -\dd < s_{j+1}-\dd) \right] \\
	=& \left[ \image\ \module{F}(s_i < s_j -\dd) \cap \ker\ \module{F}(s_j -\dd < s_{j+1}-\dd) \right].
	\end{split}
	\end{equation*}
Here the intersection is interpreted as the pullback of the two subobjects.
By a similar argument,
$$d\module{F}_\Bgroup \big( [s_{i-1},s_{j+1}) \big) - d\module{F}_\Bgroup \big( [s_{i-1},s_j) \big) = 
-\big[ \image\ \module{F}(s_{i-1} < s_j -\dd ) \cap \ker\ \module{F}(s_j -\dd < s_{j+1}-\dd) \big]~.$$
Note that
$$\image\ \module{F}(s_{i-1} < s_j -\dd ) \cap \ker\ \module{F}(s_j -\dd < s_{j+1}-\dd)$$
is a subobject of
$$\image\ \module{F}(s_i < s_j -\dd) \cap \ker\ \module{F}(s_j -\dd < s_{j+1}-\dd).$$
Therefore 
$$\module{F}_\Bgroup\big( [s_i,s_j) \big)  = \left[ \frac{\image\ \module{F}(s_i < s_j -\dd) \cap \ker\ \module{F}(s_j -\dd < s_{j+1}-\dd)}
{\image\ \module{F}(s_{i-1} < s_j -\dd ) \cap \ker\ \module{F}(s_j -\dd < s_{j+1}-\dd)} \right] \succeq [e].$$

Suppose $I = [s_i,\infty)$.
Then by a similar argument using Equation \ref{eq:halfsquare}, we have
$$\module{F}_\Bgroup\big( [s_i,\infty) \big)  = \left[ \frac{\image\ \module{F}(s_i < s_{n+1})}
{\image\ \module{F}(s_{i-1} < s_{n+1} ) } \right] \succeq [e].$$
\end{proof}

\begin{ex}
\label{ex:pmod_0}
See Figure \ref{fig:finset} for an example of a persistence module in $\Pmod(\Finset)$ and
its type $\Agroup$ persistence diagram.
Note that $\Finset$ is not an abelian category so it does not have a type $\Bgroup$ persistence
diagram.
	\begin{figure}
	\centering
	\subfigure[Persistence module]{\includegraphics[scale=0.75]{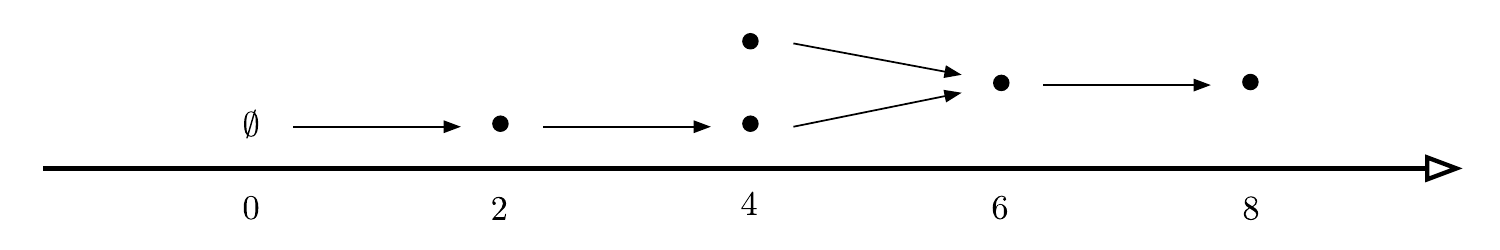}}\\ 
	\subfigure[Type $\Agroup$ persistence diagram]{\includegraphics[scale=0.75]{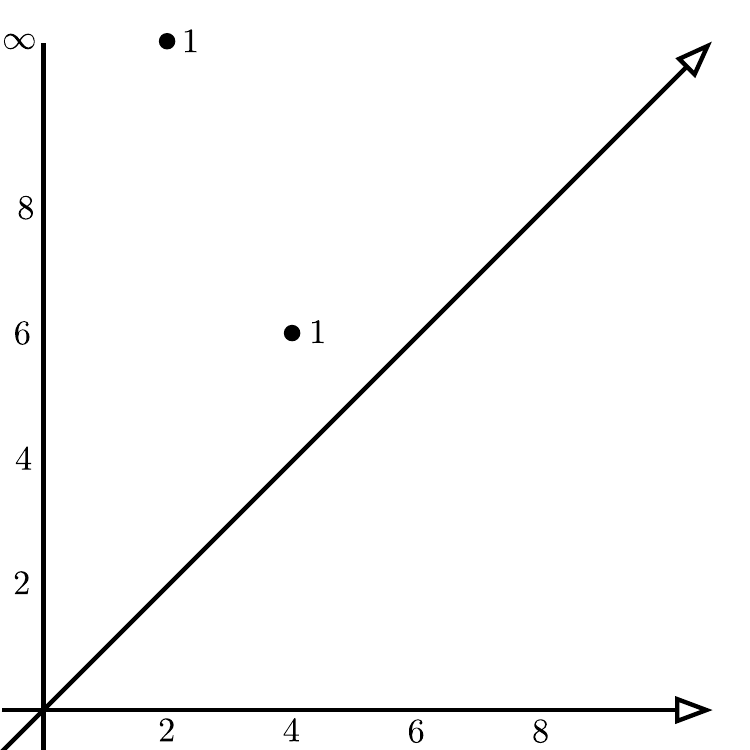}}\\
	\caption{Here we have an example of a persistence module in $\Pmod(\Finset)$ and its type $\Agroup$ persistence diagram.}
	\label{fig:finset}
	\end{figure}
\end{ex}

\begin{ex}
\label{ex:pmod_1}
See Figure~\ref{fig:vect} for an example of a persistence module in $\Pmod(\Vect)$ and its type $\Agroup$
and type $\Bgroup$ persistence diagrams.
Note that the quotient map $\pi : \Agroup(\Vect) \to \Bgroup(\Vect)$ is an isomorphism and therefore
the two diagrams are the same.

	\begin{figure}
	\centering
	\subfigure[Persistence module]{\includegraphics[scale=0.75]{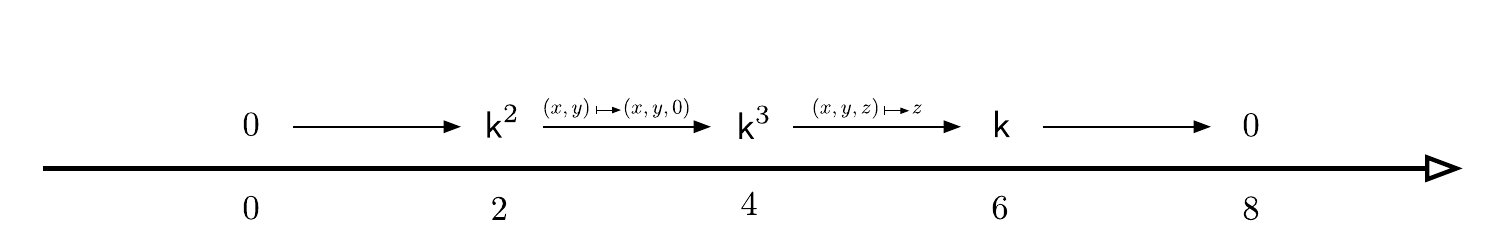}}\\ 
	\subfigure[Type $\Agroup$ and $\Bgroup$ persistence diagrams]{\includegraphics[scale=0.75]{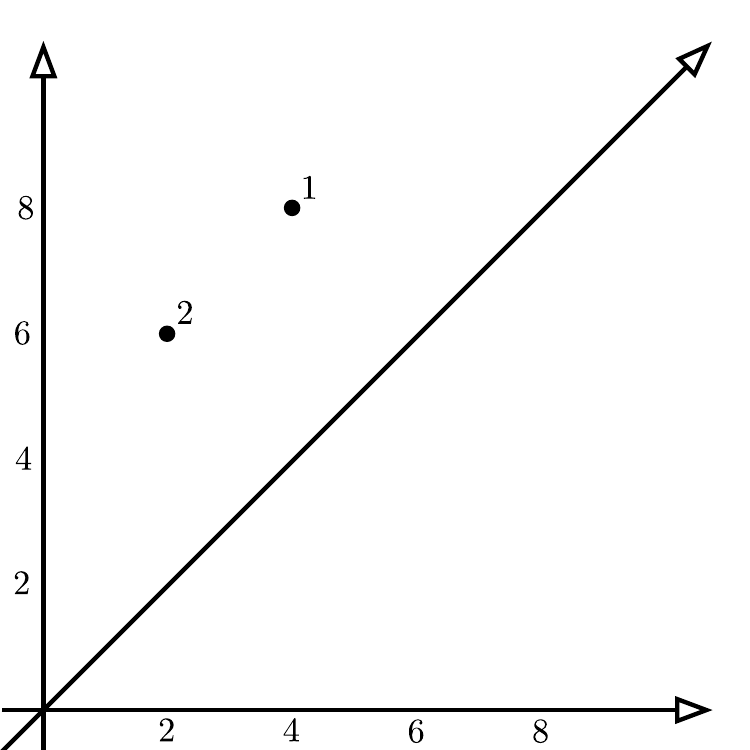}}\\
	\caption{Here we have an example of a persistence module in $\Pmod(\Vect)$ and its type $\Agroup$ and $\Bgroup$ persistence diagrams.}
	\label{fig:vect}
	\end{figure}
\end{ex}

\begin{ex}
\label{ex:pmod_2}
See Figure \ref{fig:ab} for an example of a persistence module in $\Pmod(\Ab)$ and its type
$\Agroup$ persistence diagram. Note that the quotient map $\pi: \Agroup(\Ccat) \to \Bgroup(\Ccat)$
forgets torsion and therefore the type $\Bgroup$ persistence diagram is, for this example, zero.
	\begin{figure}
	\centering
	\subfigure[Persistence module] {\includegraphics[scale=0.75]{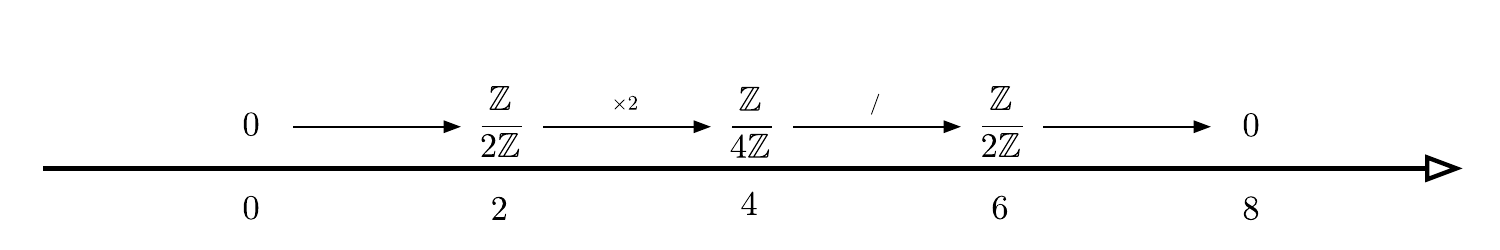}}\\ 
	\subfigure[Type $\Agroup$ persistence diagram]{\includegraphics[scale=0.75]{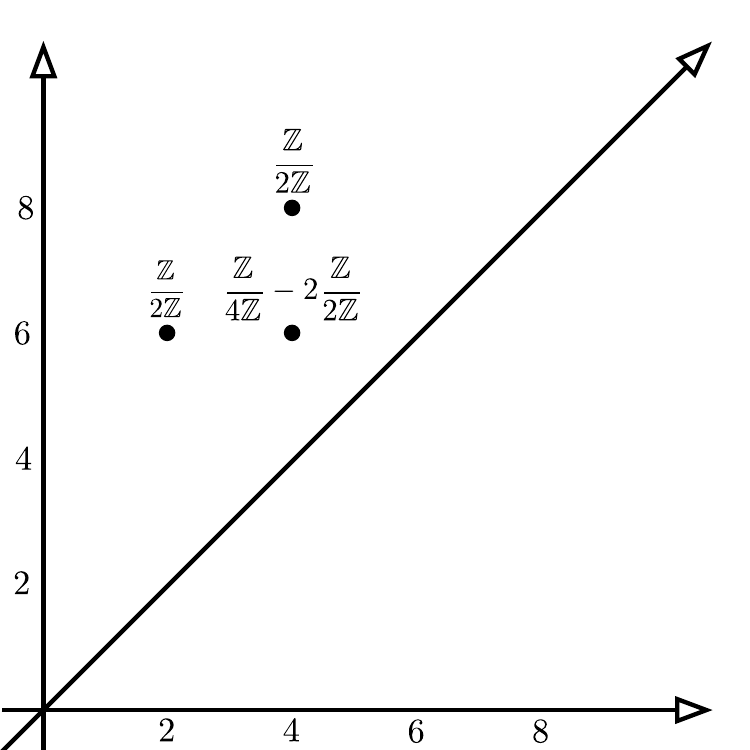}}\\
	\caption{Here we have an example of a persistence module in $\Pmod(\Ab)$ and its type $\Agroup$ persistence diagram.
	The map from $4$ to $6$ is the quotient of $\sfrac{\Zspace}{4\Zspace}$
	by the image of the previous map.}
	\label{fig:ab}
	\end{figure}
\end{ex}

\begin{ex}
See Figure \ref{fig:finab} for an example of a persistence module in $\Pmod(\Finab)$ and its type $\Agroup$
and type $\Bgroup$ persistence diagrams.
	\begin{figure}
	\centering
	\subfigure[Persistence module] 
	{\includegraphics[scale=0.75]{mod_2.pdf}}\\
	\subfigure[Type $\Agroup$ persistence diagram]{\includegraphics[scale=0.75]{ex_ab_1}} \goodgap
	\subfigure[Type $\Bgroup$ persistence diagram]{\includegraphics[scale=0.75]{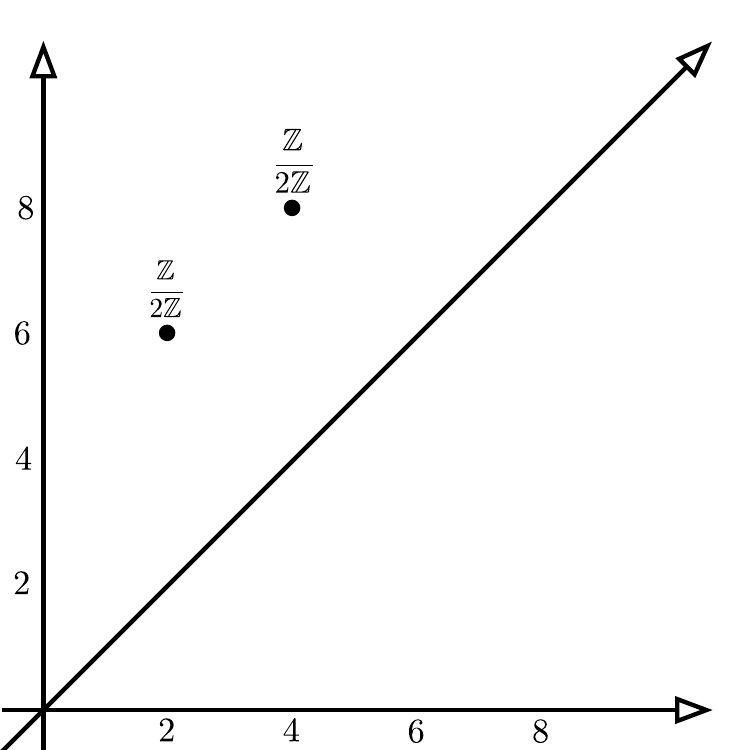}}
	\caption{Here we have an example of a persistence module in $\Pmod(\Finab)$ and its type $\Agroup$ and type $\Bgroup$ persistence diagrams.
	This is the same example module as in Figure \ref{fig:ab}.}
	\label{fig:finab}
	\end{figure}
\end{ex}

\begin{ex}
\label{ex:pmod_3}
See Figure \ref{fig:repn} for an example of a persistence module in $\Pmod \big( \Rep(\Nspace) \big)$ and its
type $\Agroup$ and type $\Bgroup$ persistence diagrams.
	\begin{figure}
	\centering
	\subfigure[Persistence module]{\includegraphics[scale=0.75]{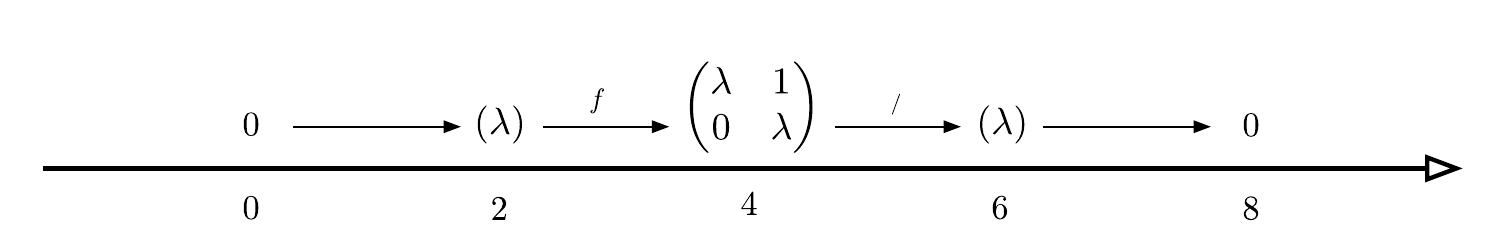}}\\ 
	\subfigure[Type $\Agroup$ persistence diagram]{\includegraphics[scale=0.75]{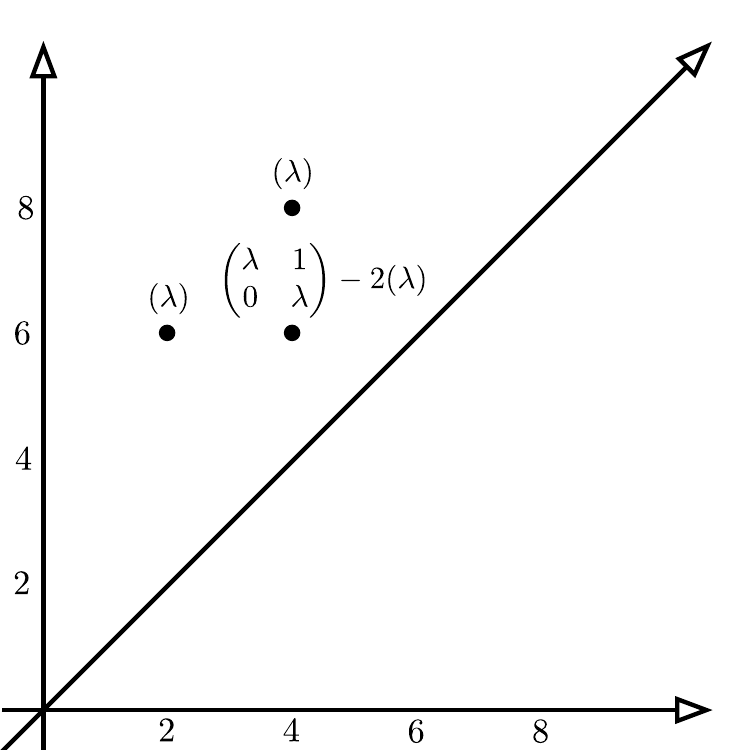}} \goodgap
	\subfigure[Type $\Bgroup$ persistence diagram]{\includegraphics[scale=0.75]{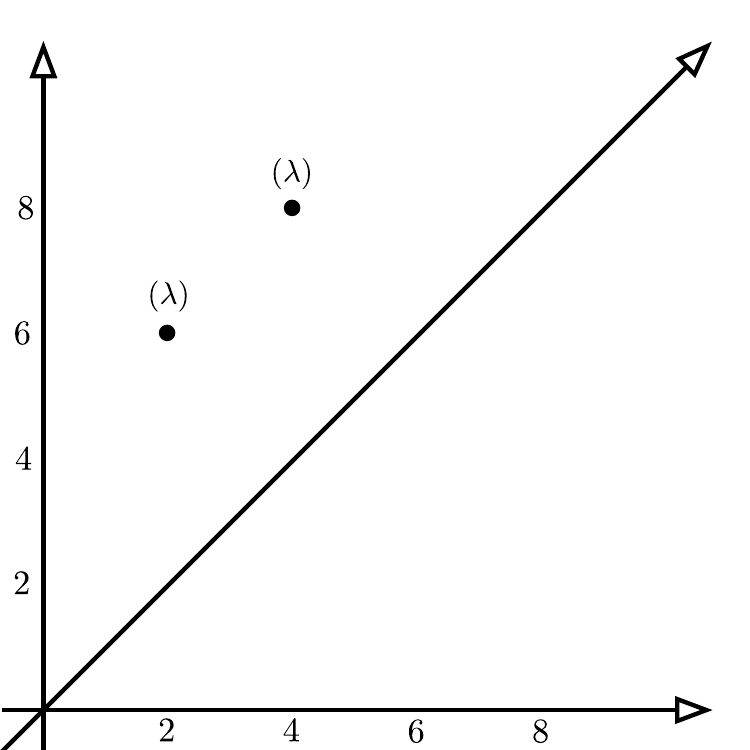}}\\
	\caption{Here we have an example of a persistence module in $\Pmod(\Ab)$ and its type $\Agroup$ 
	and type $\Bgroup$ persistence diagrams.
	The map from $4$ to $6$ is the quotient by the image of $f$.}
	\label{fig:repn}
	\end{figure}
\end{ex}

\section{Stability}
We now relate the interleaving distance between persistence modules to the erosion
distance between their persistence diagrams. 

For the first theorem, we make a simplifying assumption on $\Ccat$ that makes it possible to chase diagrams.
We assume that $\Ccat$ is concrete and that its images are concrete.
That is, $\Ccat$ embeds into the category $\Set$ and an image of a morphism in $\Ccat$ is the image 
of the corresponding set map.
Note that all our examples satisfy this criteria. 
By the Freyd-Mitchell embedding theorem \cite[page 28]{Weibel_homo}, an essentially small abelian category $\Ccat$
embeds into the category of $R$-modules, for some ring $R$, and the image of a morphism in $\Ccat$
is the image under the corresponding set map.
Therefore, all essentially small abelian categories satisfy our criteria.

\begin{thm}[Semicontinuity]
\label{thm:one}
Let $\Ccat$ be an essentially small symmetric monoidal category with images.
Suppose $\module{F} \in \Pmod(\Ccat)$ is $S = \{s_1 < \cdots < s_n \}$-constructible and
let
	$$\rho = \frac{1}{4} \min_{1< i \leq n } (s_{i} - s_{i-1}).$$
For any second persistence module $\module{G} \in \Pmod(\Ccat)$ such that $\ee = \Dist_I(\module{F}, \module{G}) < \rho$,
there is a morphism
	$$\nabla^\ee \big( \module{F}_\Agroup \big) \to \module{G}_\Agroup$$
in $\PDgm \big( \Agroup(\Ccat) \big)$.
\end{thm}

\begin{proof}
Let $\phi : \module{F} \to \Delta^\ee(\module{G})$ and 
$\psi : \module{G} \to \Delta^\ee(\module{F})$
be an $\ee$-interleaving.
For each $I \in \Dgm$ such that $\module{F}_\Agroup (I) \neq [e]$, 
we must show 
	$$d \module{F}_\Agroup \circ \Grow^\ee (I) \preceq d\module{G}_\Agroup (I).$$
By constructibility, it is sufficient to show this inequality for $I = [s_i+\ee, s_j-\ee)$ and $I = [s_i+\ee, \infty)$.
Suppose $I = [s_i+\ee,s_j-\ee)$.
Consider the following commutative diagram:
	\begin{gather}
	\begin{aligned}
	\xymatrix{
	\module{F}(s_i) \ar[d]^{\phi(s_i)} \ar[rrrrr]^{\module{F}( s_i < s_j -\dd )}
	&&&&& \module{F}(s_j -\dd)  \\ 
	\module{G}(s_i + \ee ) \ar[d]^{\psi(s_i +\ee)} \ar[rrrrr]^{\module{G}(s_i +\ee  < s_j - \ee -\dd)} &&&&& \module{G}(s_j - \ee - \dd) 
	\ar[u]^{\psi(s_j -\ee -\dd)} \\
	\module{F}(s_i + 2\ee) \ar[rrrrr]^{\module{F}(s_i + 2 \ee < s_j - 2\ee -\dd)}
	&&&&& \module{F}(s_j  - 2\ee - \dd) \ar[u]^{\phi(s_j -2\ee -\dd)}.
	}
	\end{aligned}
	\label{diagram:interleaving}
	\end{gather}
By $S$-constructibility of $\module{F}$, the two vertical compositions are isomorphisms.
By a diagram chase, we see that
$$d \module{F}_\Agroup \big( [s_i, s_j) \big) = 
d \module{G}_\Agroup \big( [s_i + \ee, s_j - \ee) \big).$$
This proves the claim.
Suppose $I$ is of the form $[s_i,\infty)$, then 
$$ d \module{F}_\Agroup \big( [s_i, \infty) \big) = d \module{G}_\Agroup \big( [s_i + \ee, \infty] \big)$$
by a similar commutative diagram.
\end{proof}
 
Semicontinuity is saying there is an open neighborhood of $\module{F}$ in the metric space of persistence modules such that for each $\module{G}$ in this open neighborhood, 
$\module{F}_\Agroup$ lives on in $\module{G}_\Agroup$.
However, semicontinuity is unsatisfying in two interesting ways.
First, the $\ee$ must be smaller than $\rho$ which is half the injectivity radius of $S$ in $\Rspace$.
Second, $\nabla^\ee \big(\module{F}_\Agroup \big) \to \module{G}_\Agroup$ but we can not prove the converse
$\nabla^\ee \big( \module{G}_\Agroup \big) \to \module{F}_\Agroup$.
The fundamental limitation here is that not all short exact sequences in $\Ccat$ split.

\begin{thm}[Continuity]
\label{thm:two}
Let $\Ccat$ be an essentially small, concrete, abelian category.
For any two persistence modules $\module{F}, \module{G} \in \Pmod(\Ccat)$, we have 
$$\Dist_E\big( \module{F}_\Bgroup, \module{G}_\Bgroup \big) \leq \Dist_I(\module{F}, \module{G}).$$
\end{thm}

\begin{proof}
Let $\ee = \Dist_I(\module{F}, \module{G})$.
For each $I \in \Dgm$ such that $\module{F}_\Agroup(I) \neq [e]$, we must show 
	$$d \module{F}_\Agroup \circ \Grow^\ee (I) \preceq d\module{G}_\Agroup (I)$$
and for each $I \in \Dgm$ such that $\module{G}_\Bgroup(I) \neq [e]$, 
we must show
	$$d \module{G}_\Agroup \circ \Grow^\ee (I) \preceq d\module{F}_\Agroup (I).$$
We will prove the first inequality and the second inequality follows 
by simply interchanging the roles of $\module{F}$ and $\module{G}$ in the proof.

Suppose $\module{F}$ is $S = \{s_1 < \cdots < s_n \}$-constructible.
By constructibility, it is sufficient to show the first inequality for $I$ of the form $[s_i+\ee, s_j-\ee)$ and $[s_i+\ee, \infty)$.
Suppose $I = [s_i+\ee,s_j-\ee)$.
Let $\phi : \module{F} \to \Delta^\ee(\module{G})$ and $\psi : \module{G} \to \Delta^\ee(\module{F})$
be an $\ee$-interleaving.
Consider the following commutative diagram:
	\begin{gather}
	\begin{aligned}
	\xymatrix{
	\module{F}(s_i) \ar[d]^{\phi(s_i)} \ar[rrrrr]^{\module{F}( s_i < s_j -\dd )}
	&&&&& \module{F}(s_j -\dd)  \\ 
	\module{G}(s_i + \ee ) \ar[rrrrr]^{\module{G}(s_i +\ee  < s_j - \ee -\dd)} &&&&& \module{G}(s_j - \ee - \dd).
	\ar[u]^{\psi(s_j -\ee -\dd)}
	}
	\end{aligned}
	\label{diagram:interleaving}
	\end{gather}
By commutativity,
	$$\image\ \module{F}(s_i < s_j - \dd) \cong \frac{\image\ \module{G}(s_i + \ee < s_j - \ee - \dd)}{
	\image\ \module{G}(s_i + \ee < s_j - \ee - \dd) \cap \ker\; \psi(s_j-\ee-\dd)}.$$
Therefore
	\begin{equation*}
	\begin{split}
	d \module{F}_\Bgroup \big( [ s_i < s_j ) \big) &= d \module{G}(s_i + \ee < s_j - \ee) - 
	[ \ker\ \psi(s_j-\ee-\dd)] \\
	&\preceq d \module{G}_\Bgroup \big( [s_i +\ee  < s_j - \ee ) \big)
	\end{split}
	\end{equation*}	
This proves the claim.
Suppose $I = [s_i,\infty)$.
Then
	 $$d \module{F}_\Bgroup \big( [ s_i < \infty ) \big) \preceq  d \module{G}_\Bgroup \big( [s_i +\ee  < \infty) \big).$$
by a similar commutative diagram.
\end{proof}

\section{Concluding Remarks}
\label{sec:end}

\paragraph{Torsion in data.}
We hope our theory will allow for the study of torsion in data.
For example, let $P \subset \Rspace^n$ be a finite set of points.
Let $f : \Rspace^n \to \Rspace$ be a function dependent on $P$, 
for example $f(x) = \min_{p \in P} || x - p ||_2$.
Apply homology with integer coefficients to the sublevel set filtration induced by $f$
and we have a constructible persistence module $\module{F} \in \Pmod(\Ab)$.
Its type $\Agroup$ persistence diagram is measuring torsion in data and semicontinuity applies.
If continuity is required, then we may look at the type $\Bgroup$ persistence diagram of $\module{F}$.
However, the type $\Bgroup$ persistence diagram forgets all torsion.
Perhaps a better approach is to apply homology with coefficients in a finite abelian group.
Then the resulting persistence module is in $\Pmod(\Finab)$ and its type $\Bgroup$
diagram encodes simple torsion.

\paragraph{Time series.}
The flexibility we offer in choosing $\Ccat$ should allow for the encoding of more structure in data.
Consider time series data.
Suppose $P = \{ p_1, \cdots, p_{k} \}$
is a finite sequence of points in $\Rspace^n$.
There is more to $P$ than its shape.
The forward shift $p_i \to p_{i+1}$ along the sequence should induce dynamics on the shape of $P$ at each scale.
The algebraic object of study is not clear, but it will certainly have more structure than a vector space
or an abelian group.

\paragraph{Non-constructible modules.}
Suppose we are given an infinite set of points $P \subset \Rspace^n$.
Then the resulting persistence module, as constructed above, is not constructible.
Is there a persistence diagram for a non-constructible persistence module?

This question is addressed by \cite{crazy_persistence} for $\Ccat = \Vect$.
They define a persistence diagram for a non-constructible persistence module as a
\emph{rectangular measure} $\mu : \Rect \to \Nspace$,
where $\Rect$ is the poset of all pairs $J \supset I$ in $\Dgm$, satisfying a certain additivity condition.
Our type $\Bgroup$ diagram should generalize to a rectangular measure.
For $\Ccat$ abelian, we may use an argument similar to the one in the proof of Proposition \ref{prop:pos} 
to assign an element of $\Bgroup(\Ccat)$ to each $J \supset I$ without making use of constructibility.
Is this assignment a rectangular measure?

\section*{Acknowledgements}
We thank Robert MacPherson for his mentorship and support.
We thank Vin de Silva for detailed comments on earlier versions of this paper.
We also thank the participants of the MacPherson Seminar on applied topology
for listening and providing helpful feedback.
Finally, we thank our anonymous reviewers for their patience and transformative feedback.

{\bibliography{ref.bib}}
\bibliographystyle{alpha}

\newpage
\appendix
\section{Krull-Schmidt}
\label{sec:krull_schmidt}
We now provide a compact treatment of Krull-Schmidt categories.
The following ideas are classical and may be found in many books, for example~\cite{anderson_fuller}.

A category $\Ccat$ is \emph{additive} if all its hom-sets are abelian, composition is bilinear, and finite products
and finite coproducts are the same. 
The (co)product of the empty set is the \emph{zero object} of $\Ccat$.
Suppose $\Ccat$ is additive.

\begin{defn}
A non-zero object $a \in \Ccat$ is \define{indecomposable} if it is not the direct sum of two non-zero objects.
\end{defn}

\begin{defn}
An additive category $\Ccat$ is \define{Krull-Schmidt} if each object $a \in \Ccat$ is isomorphic to a finite direct sum
$a \cong a_1 \oplus a_2 \oplus \cdots \oplus a_n$ and each ring of endomorphisms $\End_\Ccat(a_i)$ is \emph{local}.
That is, $0 \neq 1$ and if $f_1 + f_2  = 1$, then $f_1$ or $f_2$ is invertible.
\end{defn}

Suppose $\Ccat$ is Krull-Schmidt.

\begin{prop}
\label{prop:local_ring}
An object $a \in \Ccat$ is indecomposable iff its endomorphism ring $\End(a)$ is local.
\end{prop}
\begin{proof}
Suppose $a \in \Ccat$ is decomposable.
That is, there is an isomorphism $i :a \to a_1 \oplus a_2$ such that $a_1,a_2 \neq 0$.
Define $\pi_1 : a_1 \oplus a_2 \to a_1 \oplus a_2$ 
as the endomorphism that sends the first factor to zero
and $\pi_2 : a_1 \oplus a_2 \to a_1 \oplus a_2$ as the endomorphism that sends the second factor to zero.
Then the two maps $\rho_1, \rho_2 : a \to a$, where $\rho_1 = i^{-1} \circ \pi_1 \circ i$ and $\rho_2 = i^{-1} \circ \pi_2 \circ i$,
are both non-isomorphisms in $\End_\Ccat(a)$.
However, $\rho_0 + \rho_1 :a \to a$ is an isomorphism.
We have a contradiction of locality.

Suppose $a \in \Ccat$ is indecomposable.
Then, by definition of a Krull-Schmidt category, $\End_\Ccat(a)$ is a local ring.
\end{proof}

\begin{prop}
Each object $a \in \Ccat$ is isomorphic to a finite direct sum of indecomposables.
\end{prop}
\begin{proof}
By definition of a Krull-Schmidt category, $a \cong a_1 \oplus a_2 \oplus \cdots \oplus a_n$ where each
$\End_\Ccat(a_i)$ is a local ring.
By Proposition \ref{prop:local_ring}, each $a_i$ is indecomposable.
\end{proof}

\begin{thm}[Krull-Schmidt]
Suppose an object $c \in \Ccat$ is isomorphic to $a_1 \oplus a_2 \oplus \cdots \oplus a_m$ and $b_1 \oplus b_2 \oplus \cdots \oplus b_n$,
where each $a_i$ and $b_j$ are indecomposable.
Then $m = n$, and there is a permutation $p: [m] \to [n]$ such that $a_i \cong b_{p(i)}$.
\end{thm}
\begin{proof}
By definition of an additive category, 
we have canonical projections $\pi_i : \oplus_i a_i \to a_i$ and $\rho_j : \oplus_j b_j \to b_j$ and canonical inclusions
$\mu_i : a_i \to \oplus_i a_i$ and $\nu_j : b_j \to \oplus_j b_j$.
Furthermore $\mu_j \circ \pi_i$ and
$\nu_j \circ \rho_i$ are the identity on $a_i$ and $b_i$, respectively, iff $i = j$.
Let $f : a_1 \oplus a_2 \oplus \cdots \oplus a_m \to b_1 \oplus b_2 \oplus \cdots \oplus b_n$ be an isomorphism.

Define $h_j : a_1 \to a_1$ as $h_j = \pi_1 \circ f^{-1} \circ   \nu_j \circ \rho_j  \circ f \circ \mu_1$.
Let $h = \sum_j h_j: a_1 \to a_1$.
Observe $h$ is an isomorphism.
By locality, there is an index $j$ such that $h_j$ is an isomorphism.
This means $a_1 \cong b_j$ and we specify $p(1) = j$.
Quotient by $a_1$ and $b_j$. Repeat.
\end{proof}

\end{document}